\documentclass[english,openary,11pt]{article}
\usepackage{cmap}
\usepackage[english]{babel}
\usepackage{graphicx}
\DeclareGraphicsExtensions{.png}
\usepackage{floatflt}
\usepackage{amsfonts}
\usepackage{amsthm}
\usepackage{amsmath}
\usepackage{mathtools}                 
\usepackage{gensymb}
\usepackage{amssymb,amsfonts}
\usepackage[usenames]{color}
\usepackage{wrapfig}
\usepackage{natbib}
\usepackage{url}
\usepackage{cmap}
\usepackage{graphicx}
\usepackage{multirow}
\usepackage{comment}
\usepackage{caption}
\usepackage{subcaption}
\usepackage{float}
\usepackage{diagbox}
\usepackage{enumerate}
\usepackage[shortlabels]{enumitem}
\usepackage{titling}
\usepackage{makecell}
\usepackage{booktabs}
\usepackage{amsthm}
\usepackage{bm}
\usepackage{subcaption}

\usepackage[hidelinks,hyperindex,breaklinks]{hyperref}
\usepackage[nottoc,numbib]{tocbibind} 
\setlist[itemize]{noitemsep, topsep=3pt, leftmargin=18pt}
\setlist[enumerate]{noitemsep, topsep=3pt, leftmargin=22pt}
\usepackage[title]{appendix}
\usepackage{breakcites}
\usepackage{dashrule}
\usepackage{parcolumns}

\allowdisplaybreaks

\usepackage{tikz}
\usetikzlibrary{positioning, fit, arrows, arrows.meta,
calc, decorations.pathreplacing, backgrounds}
\usepackage{pgfplots}

\usepackage{soul}               
\setul{}{0.6pt}
\setstcolor{blue}

\usepackage[titles]{tocloft}

\newcommand\xqed[1]{%
  \leavevmode\unskip\penalty9999 \hbox{}\nobreak\hfill
  \quad\hbox{#1}}
\newcommand\demo{\xqed{$\triangle$}}
\newcommand{\demoo}{\tag*{$\triangle$}}

\usepackage[left=1.65cm,right=1.65cm,
    top=1.6cm,bottom=1.6cm,bindingoffset=0cm]{geometry}

\theoremstyle{definition}
\newtheorem{remark}{Remark}

\newtheorem{example}{Example}
\newtheorem{definition}{Definition}

\newtheorem{theorem}{Theorem}
\newtheorem{lemma}{Lemma}
\newtheorem*{auxiliary-lemma}{Auxiliary lemma}
\newtheorem{corollary}{Corollary}
\newtheorem{proposition}{Proposition}

\newtheorem*{remark-conjecture}{Remark-conjecture}

\newcommand{\R}{\mathbb{R}}
\newcommand{\N}{\mathbb{N}}

\newcommand{\E}{\mathbb{E}}
\renewcommand{\P}{\mathbb{P}}

\begin{document}

\author{Ivan Novikov, Université Paris 1 Panthéon-Sorbonne}
\title{Strategies in POMDPs with Stage Duration}
\date{}

\maketitle

\begin{abstract}
Partially observable Markov decision processes (POMDPs) with stage duration provide a framework for approximating continuous-time behavior by scaling transition probabilities with a stage duration parameter $h \in (0,1]$. While previous literature has primarily focused on the limit of the discounted value as the stage duration $h$ vanishes, this paper investigates the global behavior of the asymptotic value, $V(h)$, across varying stage durations. Our main result demonstrates that any strategy in a POMDP with stage duration $h$ can be mimicked in the base POMDP ($h=1$). Specifically, we provide an explicit construction showing that for any strategy in the POMDP with stage duration $h$, there exists a strategy in the base POMDP that secures the same asymptotic payoff. As a consequence of this theorem, we establish that the value function $V(h)$ is nondecreasing with respect to $h$, and that the continuous-time limit $\lim_{h \to 0} V(h)$ exists.
\end{abstract}

\tableofcontents


\newpage 

\textbf{Note:} Throughout the paper, we use $\triangle$ to mark the end of a definition or a remark.

\section{Introduction}

\textbf{Partially observable Markov decision processes (POMDPs)} were introduced by \cite{Dra62}. In this paper, we consider the special case where the signal depends only on the current state and is deterministic. Such a POMDP proceeds in discrete time as follows. In the beginning, the initial state is drawn according to some probability distribution, after which a deterministic signal is generated based on that state. At each stage, the decision maker observes the signal, recalls all previous actions and signals, and chooses an action. This action determines the stage payoff. The next state is then drawn randomly, conditional on the current state and action. The goal of the decision maker is to maximize his overall payoff. 

POMDPs can be generalized to the case of two decision makers with opposite interests. In this case, we refer to them as zero-sum stochastic games. In this framework, we can distinguish between the cases of full state observation, where the state is fully observed by both decision makers (introduced by \cite{Sha53}); public signals, where the signal is seen by both decision makers (see, e.g., \cite{Zil16}); and private signals, where each decision maker observes his own signal (see, e.g., \cite{Ren06}). Note that in POMDPs, there is no difference between public and private signals because there is only one decision maker.

\textbf{POMDPs with stage duration} were introduced by \cite{Ney13} in the context of zero-sum stochastic games with fully observed state. Given a stochastic game $G_1$,  \cite{Ney13} considers a family $G_h$ of stochastic games in which the leaving probabilities\footnote{A \emph{leaving probability} refers to any transition probability between two \emph{distinct} states; see Definition~\ref{klp01} for the precise definition.} and the discount rate are normalized at each stage; that is, they are proportional to $h$. Let $V_{\lambda}(h)$ denote the value of the game with discount factor $\lambda$ and stage duration $h$. The games with stage duration $h$ approximate continuous-time stochastic games as $h$ vanishes; see, e.g., the PhD thesis \cite[Introduction, \S~2.2]{Nov25a}. Previous papers on stage duration have examined the limit of $V_{\lambda}(h)$ as $h$ vanishes:
\begin{enumerate}
	\item In the case of fully observed state: \cite{Ney13}, \cite{SorVig16};
	\item In the case of public signals: \cite{Sor18}, \cite{Nov24a}, \cite{Nov24b}; 
	\item In the case of private signals: \cite{CarRaiRosVie16}, \cite{Fab16}.
\end{enumerate}
Note that not all of the above papers refer to such games as ``games with stage duration''; instead, they are sometimes called, e.g., ``games with frequent actions'', but the underlying idea is always to approximate continuous-time stochastic games using discrete-time ones.

Now, we define
$$V(h) = \lim_{\lambda \to 0} V_\lambda (h).$$
The \textbf{main goal of this article} is to investigate the global behavior of $V(h)$ across varying stage durations, rather than solely focusing on its limit as $h$ vanishes. Our main result, Theorem~\ref{mainlemma}, establishes that any strategy executed in $G_h$ can be perfectly mimicked in the base POMDP $G_1$. 
Specifically, if a strategy $\sigma$ in $G_h$ yields an asymptotic payoff of $v$, we provide an explicit construction for a strategy $\widehat \sigma$ in $G_1$ that secures the same asymptotic payoff $v$.

Theorem~\ref{mainlemma} has several corollaries: 
\begin{enumerate}
	\item Corollary~\ref{corollary1}: Any strategy in $G_{h_1}$ can be mimicked by a strategy in $G_{h_2}$, provided that $h_1 < h_2$; 
	\item Corollary~\ref{corollary2}: The function $h \mapsto V(h)$ is nondecreasing;
	\item Corollary~\ref{corollary3}: The continuous-time limit $\lim_{h \to 0} V(h)$ exists.
\end{enumerate}

In addition to the above results, we also obtain the following findings:
\begin{enumerate}
	\item In \S\ref{final1}, we consider the case of a fully observed state. We show that in this setting, the function $h \mapsto V(h)$ is constant.
	\item In \S\ref{final2}, we study the continuity of the function $h \mapsto V(h)$. Propositions~\ref{graa1} and \ref{graa2} show that this function is lower semi-continuous and left-continuous on $(0,1)$. In Example~\ref{exam1}, we construct a POMDP in which $V(h)$ is neither lower semi-continuous nor left-continuous at $h=1$.
	\item In \S\ref{grana1}, we prove that the result of Corollary~\ref{corollary2} holds true only in the case of deterministic signals. As soon as the signals are random, the opposite is often true: when $h$ is small, the decision maker can deduce the current state by observing the sequence of signals. Consequently, the function $h \mapsto V(h)$ is almost never nondecreasing when random signals are present. In Example~\ref{exam2}, we construct a POMDP that illustrates this phenomenon.
\end{enumerate}

All of our results are novel. We offer the following remarks regarding their contribution compared to the existing literature:
\begin{enumerate}
	\item Our paper examines not only the case of a vanishing stage duration $h \to 0$, but also that of a fixed stage duration $h \in (0,1]$. While the latter case is well understood in the case of a fully observed state, to the best of our knowledge, it has not yet been studied for a partially observed state. This paper fills this gap by initiating research in this direction.
		\item Corollary~\ref{corollary3} considers the continuous-time limit $\lim_{h \to 0} V(h)$ for a patient decision maker (i.e. as the discount factor $\lambda$ vanishes). While previous literature has also studied the limit $\lim_{h \to 0} V(h)$, it considered the case of a fixed discount factor $\lambda$, unlike our approach.
		\item Note that in the context of stage duration, the asymptotic case $\lambda \to 0$ was also considered in \cite{Nov24b}. However, while that paper considered the double limit, $\lim_{\lambda \to 0} \lim_{h \to 0} V_\lambda(h),$ the present paper considers the limit $\lim_{\lambda \to 0} V_\lambda(h)$ for a fixed $h$, as well as the double limit $\lim_{h \to 0} \lim_{\lambda \to 0} V_\lambda(h)$. Since there is no a priori reason to believe that the equality
$$\lim_{\lambda \to 0} \lim_{h \to 0} V_\lambda(h) = \lim_{h \to 0} \lim_{\lambda \to 0} V_\lambda(h)$$
holds, the results of our paper are independent of those in \cite{Nov24b}.
\end{enumerate}

\section{POMDPs with stage duration}

A \emph{partially observable Markov decision process (POMDP)} is a 7-tuple $(\Omega, A, S, f, g, P, p_1)$, where $\Omega$ is the finite set of states, $A$ is the finite set of actions, $S$ is the finite set of signals, $f : \Omega \to S$ is the function giving a signal for each state, $g : \Omega \times A \to \R$ is the stage payoff function, $P : \Omega \times A \to \Delta(\Omega)$ is the transition probability function, and $p_1 \in \Delta(\Omega)$ is the initial distribution on the states. The POMDP $(\Omega, A, S, f, g, P, p_1)$ proceeds as follows. An initial state $\omega_1$ is drawn from $p_1$, and the decision maker receives the signal $f(\omega_1)$. At each stage $n \in \N^*$, the decision maker chooses an action $a_n \in A$ and receives the unobserved payoff $g(\omega_n, a_n)$. The next state $\omega_{n+1}$ is drawn according to $P(\omega_n, a_n)$, and the decision maker receives the signal $f(\omega_{n+1})$.

\emph{A history of length $t \in \N$} in the POMDP $(\Omega, A, S, f, g, P, p_1)$ is
$(s_1, a_1, s_2, a_2, \ldots, s_{t-1}, a_{t-1}, s_t)$.
The set of all histories of length $t$ is 
$H_t := S \times (A \times S)^{t-1}$.
A \emph{(behavior) strategy} of the decision maker is a function 
$\sigma : \bigcup_{t \ge 1} H_t \to \Delta(A)$.
The decision maker's strategy induces a probability distribution on the set $S \times (A \times S)^{\N^*}$. (Indeed, the strategy induces a probability distribution on the set $H_1$, then on the set $H_2$, etc. By the Kolmogorov extension theorem, this probability can be extended in a unique way to the set 
$S \times (A \times S)^{\N^*}$). In particular, given an initial probability distribution $p_0 \in \Delta(\Omega)$, a strategy
$\sigma : \bigcup_{t \ge 1} H_t \to \Delta(A)$, 
and the induced probability measure $\P^{p_0}_{\sigma}$ on $S \times (A \times S)^{\N^*}$, we can consider the expectation $\E^{p_0}_{\sigma}$ of any random variable defined on $\bigcup_{t \ge 1} H_t$.

\begin{definition}[POMDP with stage duration, {\cite[Definition~1]{Nov24b}} and {\cite[p. 240]{Ney13}}]
Fix a POMDP 
$$G_1 = (\Omega, A, S, f, g, P, p_1).$$
The \emph{POMDP with stage duration} $h \in (0,1]$ is the POMDP
$$G_h = (\Omega, A, S, f, g, P_h, p_1),$$
with
$$P_h (\cdot \mid \omega, a) = h P(\cdot \mid \omega, a) + (1-h) \delta_\omega (\cdot),$$
where 
$$\delta_\omega (\omega') = 
\begin{cases}
1, &\text{if } \omega' = \omega; \\
0, &\text{otherwise.}
\end{cases}$$
In $G_h$, we consider the asymptotic value
\[V(h) := \lim_{\lambda \to 0} \left[ \sup_\sigma \E_\sigma^{h} \left( \lambda h \sum_{i = 1}^{\infty} (1-\lambda h)^{i-1} g(\omega_i, a_i) \right)\right]. \demoo\]
\label{klp01}
\end{definition}

In the above definition, we use $\E_\sigma^{h}$ to denote the expectation generated by the strategy $\sigma$ in $G_h$. Formally, this expectation also depends on the initial probability distribution $p_0$, but since this distribution is independent of the stage duration $h$, we omit $p_0$ from the notation.

\begin{remark}
Definition~\ref{klp01} coincides with the one given in \cite{Nov24b} except that the latter considers the $\lambda$-discounted payoff
$$V_\lambda(h) := \sup_\sigma \E_\sigma^{h} \left( \lambda h \sum_{i = 1}^{\infty} (1-\lambda h)^{i-1} g(\omega_i, a_i) \right),$$
where $\lambda \in (0,1],$ 
whereas we consider its limit $V(h) =  \lim_{\lambda \to 0} V_\lambda(h).$ 
Definition~\ref{klp01} is also similar to the one introduced in \cite{Ney13}, except that the latter considers the case of full state observation, and uses $V_\lambda(h)$ instead of $V(h)$. \demo
\end{remark}

We will need an alternative expression for $V(h)$. To this end, we provide the following definition.

\begin{definition}[cf. {\cite[p. 286]{Ara93}}]
For each $h \in (0,1]$ and each strategy $\sigma$ in $G_h$, we define the \emph{expected long-run average payoff} as
\[R(\sigma, h) := \liminf_{T \to \infty} \E_\sigma^{h} \left( \frac{1}{T} \sum_{i = 1}^{T} g(\omega_i, a_i) \right).\demoo \]
\end{definition}

\begin{remark}
We have 
\begin{align*}
V(h) = \lim_{\lambda \to 0} \left[ \sup_\sigma \E_\sigma^{h} \left( \lambda h \sum_{i = 1}^{\infty} (1-\lambda h)^{i-1} g(\omega_i, a_i) \right)\right] =& \lim_{\lambda \to 0} \left[\sup_\sigma \E_\sigma^{h} \left( \lambda \sum_{i = 1}^{\infty} (1-\lambda)^{i-1} g(\omega_i, a_i) \right)\right] \\=& \sup_\sigma \left[\liminf_{T \to \infty} \E_\sigma^{h} \left( \frac{1}{T} \sum_{i = 1}^{T} g(\omega_i, a_i) \right)\right],
\end{align*}
where the first equality is obtained by using the change of variable $\lambda \mapsto \lambda / h$, and the last equality follows from \cite{RosSolVie02}, see also \cite[Remark~2.1]{Cha22}. \demo
\label{remark001}
\end{remark}

Remark~\ref{remark001} implies that $V(h) = \sup_\sigma R(\sigma, h).$

\section{Main results}

\begin{theorem}
Let $h \in (0,1)$ and let $\sigma$ be a strategy in $G_h$. Then there exists a strategy $\widehat \sigma$ in $G_1$ such that $$R(\sigma, h) = R(\widehat \sigma, 1).$$
\label{mainlemma}
\vspace{-0.6cm}
\end{theorem}

\begin{corollary}
Let $0 < h_1 < h_2 \le 1$ and let $\sigma$ be a strategy in $G_{h_1}$. Then there exists a strategy $\widehat \sigma$ in $G_{h_2}$ such that 
$$R(\sigma, h_1) = R(\widehat \sigma, h_2).$$
\label{corollary1}
\vspace{-0.6cm}
\end{corollary}

\begin{corollary} The function $h \mapsto V(h)$ is nondecreasing.
\label{corollary2}
\end{corollary}

\begin{corollary} The limit $\lim_{h \to 0} V(h)$ exists.
\label{corollary3}
\end{corollary}

\begin{remark} We make a few comments regarding the function $h \mapsto V(h)$.
\begin{enumerate}
	\item In the case of a fully observed state, the function $h \mapsto V(h)$ is constant (see \S\ref{final1}).
	\item The function $h \mapsto V(h)$ is lower semi-continuous on $(0,1)$ and left-continuous on $(0,1)$ (see \S\ref{final2}).
	\item It is crucial that the signal given to the decision maker is deterministic. Otherwise, the function $h \mapsto V(h)$ need not be nondecreasing (see \S\ref{grana1}). \demo
\end{enumerate}

\end{remark}

\section{Proof of Theorem~1}

\subsection{Construction of the strategy $\widehat \sigma$}

We first construct the strategy $\widehat \sigma$. We fix $h \in (0,1)$ and a strategy $\sigma$ in $G_h$.

\begin{definition} \leavevmode
\begin{enumerate}
\item We denote by $X = (X_1, X_2, \ldots)$ the stochastic process of i.i.d. random variables, where 
$$\Big(\P(X_i = 1) = h \quad \text{ and } \quad \P(X_i = 0) = 1 - h \Big) \iff X_i \sim Bernoulli(h).$$
\item We denote by $T = (T_0, T_1, T_2, \ldots)$ the stochastic process, in which $T_i$ is the random variable
\begin{align*}
	T_0 &= 0; \\
	T_i &= \inf\{n > T_{i - 1} \mid X_n = 1\} \quad \text{ for } i>0.
\end{align*}
\item We denote by $N = (N_1, N_2, \ldots)$ the stochastic process of i.i.d. random variables, where 
\[ \Big(N_i = T_i - T_{i-1}\Big) \iff \forall i \in \N^* \;\; N_i = k \text{ with probability } h (1-h)^{k-1} \iff N_i \sim Geometric(h). \demoo \]
\end{enumerate}
\end{definition}

\begin{remark}[Interpretation of $X, T,$ and $N$]
In $G_h$, the next state is chosen according to $$h P(\cdot \mid \omega, a) + (1-h) \delta_\omega (\cdot).$$
This means that at the end of each stage, the next state is drawn according to $P$ with probability $h$; otherwise, the state remains unchanged (drawn according to $\delta_\omega$) with probability $1-h$. Hence:
\begin{enumerate}
	\item $X_i$ indicates whether the transition at the end of the $i$-th stage is governed by $P$ (if $X_i = 1$) or by $\delta_\omega$ (if $X_i = 0$).
	\item $T_i$ is the stage number at which the state transition is governed by $P$ for the $i$-th time.
	\item $N_i$ is the duration of the $i$-th epoch. It represents the number of stages between the $(i-1)$-th and $i$-th time the transition is governed by $P$, consisting of $N_i - 1$ self-transitions ($\delta_\omega$) followed by one $P$-transition. \demo
\end{enumerate}
\end{remark}

\begin{remark}
This remark summarizes some useful properties of $T$ and $N$ that will be used in subsequent proofs. For $N$, we have
$$\P(N_j \ge m) = (1-h)^{m-1} \quad \text{ and } \quad \E(N_j) = \frac 1 h \quad \text{ and } \quad \operatorname{Var}(N_j) = \frac{1-h}{h^2}.$$
For $T$, we have
\[T_j = \sum_{i = 1}^{j} N_i \quad \text{ and } \quad \E(T_j) = \frac j h \quad \text{ and } \quad \operatorname{Var}(T_j) = \frac{1-h}{h^2} j. \demoo \]
\end{remark}

Fix $k \in \N^*$. For a fixed infinite history $\mathcal H = (s'_1, a'_1, s'_2, \ldots)$ in $G_h$, 
we denote the filtered history $\mathcal H^{fil}_k$ as the random vector $$\mathcal H^{fil}_k = (s'_1, a'_{T_1}, s'_{T_1 + 1}, a'_{T_2}, \ldots, a'_{T_{k-1}}, s'_{T_{k-1}+1}).$$ 

The probability measure $\P_\sigma^h$ is defined on an extended joint probability space that encompasses both the histories of $G_h$ and the auxiliary stochastic processes $X, T,$ and $N$.
Similarly, we use $\E_\sigma^h$ to denote the expectation of random variables defined on this probability space.

\begin{definition}
Let $\eta_k = (s_1, a_1, s_2, \ldots, a_{k-1}, s_{k})$ be a history of length $k$ in $G_1$, and let $\widehat a \in A$. The strategy $\widehat{\sigma}$ is defined by
\begin{align*}
\widehat{\sigma}(\eta_k)(\widehat a)
&:=
\P^h_\sigma(a'_{T_k} = \widehat a \mid \mathcal H^{fil}_k = \eta_k).
\end{align*}
If $\P^h_\sigma(\mathcal H^{fil}_k = \eta_k) = 0$, then we define $\widehat{\sigma}(\eta_k)$ as an arbitrary fixed mixed action. \demo
\end{definition}

\begin{remark}[Intuition behind the construction of $\widehat{\sigma}$]
The strategy $\widehat{\sigma}$ is designed to simulate in $G_1$ the strategy $\sigma$ from $G_h$. In $G_h$, the state only truly transitions (according to $P$) at the random stages $T_1, T_2, \ldots$. Between these stages, the state is frozen.

The decision maker in $G_1$ does not experience these frozen periods; instead, every stage in $G_1$ corresponds to a true transition. Therefore, to replicate $\sigma$, the decision maker in $G_1$ examines the current history $\eta_k$ and asks: \textit{``If I were playing $G_h$ and the sequence of true transitions perfectly matched $\eta_k$, what action would I play at the (random and unobserved) moment $T_k$ when the state is finally about to change?''}

As the durations of the frozen periods ($N_i$) and the intermediate actions taken during them are absent in $G_1$, $\widehat{\sigma}$ effectively ``filters out'' this information. It calculates the expected mixed action that $\sigma$ would play at stage $T_k$, conditioned on the fact that the filtered history $\mathcal{H}^{fil}_k$ of true transitions aligns with the history $\eta_k$ observed in $G_1$. \demo
\end{remark}

\subsection{Proof of Theorem 1 and Corollaries~\ref{corollary1}-\ref{corollary3} modulo technical lemmas}

We first state four technical lemmas.

\begin{lemma}
For each $k \in \N^*$, we have
$$\E_{\widehat \sigma}^1 g(\omega_k, a_k) = \E_\sigma^h g(\omega_{T_k}, a_{T_k}).$$
\label{lemma1}
\end{lemma}

\begin{lemma} For each $k \in \N^*$, we have
$$\E_{\sigma}^h \left( \sum_{j = T_{k-1}+1}^{T_k} g(\omega_j, a_j) \right) = \frac 1 h \E_\sigma^h g(\omega_{T_k}, a_{T_k}).$$
\label{lemma3}
\end{lemma}

\begin{lemma}
Let $t_k = \lfloor k/h \rfloor$.
For each $k \in \N^*$, we have
$$
\liminf_{k \to +\infty} \E_{\sigma}^h\left( \frac{1}{t_k} \sum_{j = 1}^{t_k} g(\omega_j, a_j) \right)
=
\liminf_{k \to +\infty} \E_{\sigma}^h\left( \frac{h}{k} \sum_{j = 1}^{T_k} g(\omega_j, a_j) \right).
$$
\label{lemma4}
\end{lemma}

\begin{lemma}
Let $M \in \N^*$. Let $\{x_n\}_{n=1}^\infty$ be a sequence and $\{x_{n_k}\}_{k=1}^\infty$ be a  subsequence of $\{x_n\}_{n=1}^\infty$. Suppose that
$$	\lim_{j \to \infty} (x_{j+1} - x_j) = 0 \qquad \text{ and } \qquad |n_{j+1} - n_j| \le M \text{ for all } j \in \N^*.$$
We then have
$$\liminf_{n \to \infty} x_n = \liminf_{k \to \infty} x_{n_k}.$$
\label{lemma5}
\end{lemma}

\begin{proof}[Proof of Theorem~\ref{mainlemma} modulo technical lemmas]
To show that $R(\sigma, h) = R(\widehat \sigma, 1)$, we manipulate the expressions for $R(\sigma, h)$ and $R(\widehat \sigma, 1)$.

We start with $R(\widehat \sigma, 1)$. By Lemma~\ref{lemma1}, we have
\begin{equation}
\begin{aligned}
R(\widehat \sigma, 1) = \liminf_{T \to +\infty} \E_{\widehat \sigma}^1 \left( \frac{1}{T} \sum_{k = 1}^{T} g(\omega_k, a_k) \right) 
&= \liminf_{T \to +\infty} \left( \frac{1}{T} \sum_{k = 1}^{T} \E_{\widehat \sigma}^1 g(\omega_k, a_k) \right) \\
&= \liminf_{T \to +\infty} \left( \frac{1}{T} \sum_{k = 1}^{T} \E_\sigma^h g(\omega_{T_k}, a_{T_k}) \right).
\end{aligned}
\label{ahaha5}
\end{equation}

We now aim to compute $R(\sigma, h).$ 

Consider the sequence $\{x_t\}_ {t=1}^\infty$ and the subsequence $\{x_{t_k}\}_ {k=1}^\infty$ of $\{x_t\}_ {t=1}^\infty$, where
$$x_t := \E_{\sigma}^h \left( \frac{1}{t} \sum_{i = 1}^{t} g(\omega_i, a_i) \right) \quad \text{ and } \quad t_k = \lfloor k/h \rfloor.$$

Let $M=\sup_{\omega, a} g(\omega, a)$. We have 
\begin{equation}
\begin{aligned}
\left|x_{t+1} - x_t\right| 
&=
\left|\E_{\sigma}^h \left( \frac{1}{t+1} \sum_{i = 1}^{t+1} g(\omega_i, a_i) - \frac{1}{t} \sum_{i = 1}^{t} g(\omega_i, a_i) \right) \right|  \\
&=
\left|\frac{-1}{t(t+1)} \sum_{i = 1}^{t} \E_{\sigma}^h g(\omega_i, a_i) + \frac{1}{t+1} \E_{\sigma}^h g(\omega_{t+1}, a_{t+1})\right| \le \frac{M t}{t(t+1)} + \frac{M}{t+1} \xrightarrow{t \to \infty} 0.
\end{aligned}
\label{ahaha0001}
\end{equation}

We also have for any $k \in \N^*$
\begin{equation}
\left|t_{k+1} - t_k\right| \le 1 + \frac 1 h.
\label{ahaha0002}
\end{equation}

Now Lemma~\ref{lemma5}, together with \eqref{ahaha0001} and \eqref{ahaha0002}, implies
\begin{equation}
R(\sigma, h) = \liminf_{T \to +\infty} \E_{\sigma}^h \left( \frac{1}{T} \sum_{i = 1}^{T} g(\omega_i, a_i) \right) = \liminf_{k \to +\infty} \E_{\sigma}^h\left( \frac{1}{t_k} \sum_{j = 1}^{t_k} g(\omega_j, a_j) \right),
\label{ahaha0}
\end{equation}

By Lemmas \ref{lemma3} and \ref{lemma4}, we have
\begin{equation}
\begin{aligned}
\liminf_{k \to +\infty} \E_{\sigma}^h\left( \frac{1}{t_k} \sum_{j = 1}^{t_k} g(\omega_j, a_j) \right)
=
\liminf_{k \to +\infty} \E_{\sigma}^h\left( \frac{h}{k} \sum_{j = 1}^{T_k} g(\omega_j, a_j) \right) 
&=
\liminf_{K \to +\infty} \E_{\sigma}^h\left( \frac{h}{K} \sum_{k = 1}^{K} \sum_{j = T_{k-1}+1}^{T_k} g(\omega_j, a_j) \right)
\\
&= \liminf_{K \to +\infty} \left( \frac{1}{K} \sum_{k = 1}^{K} \E_\sigma^h g(\omega_{T_k}, a_{T_k}) \right).
\end{aligned}
\label{ahaha2}
\end{equation}

By combining \eqref{ahaha5}, \eqref{ahaha0}, and \eqref{ahaha2}, we obtain
$$R(\sigma, h) = R(\widehat \sigma, 1).\qedhere$$
\end{proof}

\begin{proof}[Proof of Corollary~\ref{corollary1}]
Note that given $h_1, h_2$ with $0 < h_1 < h_2 < 1$, we consider $G_{h_1}$ as the POMDP with stage duration $h_1$ relative to the base POMDP $G_1$. However, we can also consider $G_{h_1}$ as the POMDP with stage duration $h_1 / h_2$ relative to the base POMDP $G_{h_2}$. Indeed, for the transition law $P_{h_1}$ in $G_{h_1}$, we have
$$P_{h_1} = (1 - h_1) Id + h_1 P_1 = \left(1 - \frac {h_1} {h_2}\right) Id + \frac{h_1}{h_2} (\left(1 - h_2\right) Id + h_2 P_1) = \left(1 - \frac {h_1} {h_2}\right) Id + \frac{h_1}{h_2} P_{h_2}.$$

Consequently, Theorem~\ref{mainlemma} implies that there is a strategy $\widehat \sigma$ in $G_{h_2}$ such that
$$R(\sigma, h_1) = R(\widehat \sigma, h_2). \qedhere$$
\end{proof}

\begin{proof}[Proof of Corollary~\ref{corollary2}]
This follows from Corollary~\ref{corollary1} and Remark~\ref{remark001}. For any $0 < h_1 \le h_2 \le 1$, we have
$$V(h_1) = \sup\{R(\sigma, h_1)) \; | \; \sigma \text{ is a strategy in }G_{h_1} \} \le  \sup\{R(\sigma, h_2) \; | \; \sigma \text{ is a strategy in }G_{h_2} \} = V(h_2). \qedhere$$
\end{proof}

\begin{proof}[Proof of Corollary~\ref{corollary3}]
This follows directly from Corollary~\ref{corollary2} and the fact that the stage payoff function $g$ is bounded.
\end{proof}

\subsection{Proof of Lemma~\ref{lemma1}}

To avoid ambiguity between the two distinct stochastic processes, we adopt the following notational convention in the proof of Lemma~\ref{lemma1}. 
The state $\omega_{k}$ (resp. action $a_{k}$, signal $s_{k}$) refers exclusively to the $k$-th stage state (resp. action, signal) in $G_{1}$.
Conversely, the state $\omega'_{k}$ (resp. action $a'_{k}$, signal $s'_{k}$) refers exclusively to the $k$-th stage state (resp. action, signal) in $G_{h}$.

In what follows, when manipulating conditional probabilities, we adopt the standard convention that $0 \times \text{undefined} = 0$. Consequently, whenever we evaluate a conditional probability of the form $\mathbb{P}(A \mid B)$, it is implicitly assumed that $\mathbb{P}(B) > 0$. In cases where $\mathbb{P}(B) = 0$, the conditional probability $\mathbb{P}(A \mid B)$ is technically undefined; however, because such terms only appear multiplied by $\mathbb{P}(B) = 0$ (such as in the law of total probability or Bayes' theorem), they contribute zero to the overall expression. Thus, we do not need to explicitly evaluate or restrict our summations to exclude these zero-probability events.

To prove Lemma~\ref{lemma1}, we first state and prove the following technical lemma.

\begin{lemma}
Let $k \in \N_0, \omega \in \Omega$, and let: 
\begin{itemize}
	\item $\eta_{k+1}$ be a history of length $k+1$ in $G_1$;
	\item $\eta'_{k+1}$ be a history of length $T_{k+1}$ in $G_h$;
	\item $\mathcal H^{fil}_{k+1}$ be the filtered history of length $k+1$ (as defined previously);
	\item $\mathcal H^{rem}_{k+1}$ be the supplementary information in $\eta'_{k+1}$ not present in $\mathcal{H}_{k+1}^{fil}$.
\end{itemize}

Then $\mathcal{H}_{k+1}^{rem}$ is conditionally independent of $\omega'_{T_{k}+1}$ given $\mathcal H^{fil}_{k+1} = \eta_{k+1}$. That is,
	\begin{align*}
		\P_{\sigma}^h (\mathcal H^{rem}_{k+1} \mid \mathcal H^{fil}_{k+1} = \eta_{k+1}, \omega'_{T_k + 1} = \omega) &= \P_{\sigma}^h (\mathcal H^{rem}_{k+1} \mid \mathcal H^{fil}_{k+1} = \eta_{k+1})\\
		&\text{and}\\
		\P_{\sigma}^h(\omega'_{T_{k}+1} = \omega \mid \mathcal{H}_{k+1}^{rem}, \mathcal H^{fil}_{k+1} = \eta_{k+1}) &= \P_{\sigma}^h(\omega'_{T_{k}+1} = \omega \mid \mathcal H^{fil}_{k+1} = \eta_{k+1}).
	\end{align*}
\label{aux-lemma}
\end{lemma}

\begin{proof}[Proof of Lemma~\ref{aux-lemma}]

We break $\mathcal{H}_{k+1}^{rem}$ into its three components:

\begin{itemize}
	\item \textbf{The waiting periods $N_{i}$}. These depend only on $h$ and are thus independent of $\omega'_{T_{k}+1}$. 
	\item \textbf{The intermediate signals} received during the waiting periods $N_{i}$. During any such period, the state remains constant.
Since the signals depend only on the state and are deterministic, the signals observed during the period $N_i$ simply repeat the $i$-th signal recorded in $\mathcal{H}_{k+1}^{fil}$. Hence, $\omega'_{T_{k}+1}$ does not add any new information.
	\item \textbf{The intermediate actions} taken during the waiting periods $N_{i}$. Because the decision maker only observes signals and not the state itself, the intermediate actions in $\mathcal H^{rem}_{k+1}$ depend only on the sequence of received signals and past actions. Since these signals are completely determined by $\mathcal{H}_{k+1}^{fil}$, the generation of $\mathcal H^{rem}_{k+1}$ is conditionally independent of $\omega'_{T_{k}+1}$.
\end{itemize}
Because all three components of $\mathcal{H}_{k+1}^{rem}$ depend only on $h$ or the information already captured in $\mathcal{H}_{k+1}^{fil}$, the true state $\omega'_{T_{k}+1}$ provides no new predictive power. Therefore, the conditional independence holds.
\end{proof}

\begin{proof}[Proof of Lemma~\ref{lemma1}]
We will prove a more general statement: for any $\omega \in \Omega, a \in A, k \in \N^*$, and any history $\eta_k$ of length $k$, we have
\begin{equation}
\P_{\widehat \sigma}^1(\eta_k, \omega_k = \omega, a_k = a) = \P_{\sigma}^h(\mathcal H^{fil}_k = \eta_k, \omega'_{T_k} = \omega, a'_{T_k} = a).
\label{rooss10}
\end{equation}

This identity implies the lemma. Indeed, by \eqref{rooss10} we have
\begin{equation}
\begin{aligned}
 \P_{\widehat \sigma}^1(\omega_k = \omega, a_k = a) &= \sum_{\substack{\text{Histories } \eta_{k} \text{ of} \\ \text{length } k \text{ in } G_1}} \P_{\widehat \sigma}^1(\eta_k, \omega_k = \omega, a_k = a)
= \sum_{\substack{\text{Histories } \eta_{k} \text{ of} \\ \text{length } k \text{ in } G_1}} \P_{\sigma}^h(\mathcal H^{fil}_k = \eta_k, \omega'_{T_k} = \omega, a'_{T_k} = a) \\
&= \P_{\sigma}^h(\omega'_{T_k} = \omega, a'_{T_k} = a).
\end{aligned}
\label{rooss101}
\end{equation}

Consequently, we obtain by \eqref{rooss101}
$$\E_{\widehat \sigma}^1 g(\omega_k, a_k) =
\sum_{\omega \in \Omega, a \in A}\P_{\widehat \sigma}^1(\omega_k = \omega, a_k = a) \cdot g(\omega, a) 
= \sum_{\omega \in \Omega, a \in A} \P_{\sigma}^h(\omega'_{T_k} = \omega, a'_{T_k} = a) \cdot g(\omega, a)
= \E_\sigma^h g(\omega'_{T_k}, a'_{T_k}).$$

The rest of the proof is dedicated to justifying \eqref{rooss10}. The proof proceeds by induction on $k$. We first consider the base case $k=1$. Since the initial probability distribution is the same in both $G_1$ and $G_h$, we have
\begin{equation}
\P_{\widehat \sigma}^1(\eta_1, \omega_1 = \omega) = \P_{\sigma}^h(\mathcal H^{fil}_1 = \eta_1, \omega_1 = \omega).
\label{tralala1}
\end{equation}

Because the state is frozen between stage $T_k + 1$ and stage $T_{k+1}$, we have for any $k$
\begin{equation}
\omega'_{T_{k+1}} = \omega'_{T_{k}+1}.
\label{roo10x3}
\end{equation}

By the definition of conditional probability, we have
\begin{equation}
\P_{\widehat \sigma}^1(\eta_1, \omega_1 = \omega, a_1 = a) = \P_{\widehat \sigma}^1(a_1 = a \mid \eta_1, \omega_1 = \omega) \cdot \P_{\widehat \sigma}^1(\eta_1, \omega_1 = \omega) = \widehat \sigma(\eta_1)(a) \cdot \P_{\widehat \sigma}^1(\eta_1, \omega_1 = \omega).
\label{tralala2}
\end{equation}

By the definition of conditional probability and by \eqref{roo10x3}, we have
\begin{equation}
\begin{aligned}
\P_{\sigma}^h(\mathcal H^{fil}_1 = \eta_1, \omega'_{T_1} = \omega, a'_{T_1} = a) 
&= \P_{\sigma}^h(a'_{T_1} = a \mid \mathcal H^{fil}_1 = \eta_1, \omega'_{T_1} = \omega) \cdot \P_{\sigma}^h(\mathcal H^{fil}_1 = \eta_1, \omega'_{T_1} = \omega) \\
&= \P_{\sigma}^h(a'_{T_1} = a \mid \mathcal H^{fil}_1 = \eta_1, \omega'_{1} = \omega) \cdot \P_{\sigma}^h(\mathcal H^{fil}_1 = \eta_1, \omega'_{1} = \omega).
\end{aligned}
\label{tralala3}
\end{equation}

Now, let us fix a history $\eta'_{1}$ of length $T_1$ in $G_h$. Provided that the event 
$$\left\{\mathcal H^{fil}_1 = \eta_1, \eta'_{1}, \omega'_{1} = \omega\right\}$$
has positive probability, we have $\eta'_{1} = \left(\mathcal H^{fil}_1, \mathcal H^{rem}_1\right)$. Thus, by the law of total probability, we obtain
\begin{equation}
\begin{aligned}
&\P_{\sigma}^h (a'_{T_1} = a \mid \mathcal H^{fil}_1 = \eta_1, \omega'_{1} = \omega) 
\\
=&
\sum_{\substack{\text{Histories } \eta'_{1} \text{ of} \\ \text{length } T_{1} \text{ in } G_h}}
\P_{\sigma}^h (a'_{T_1} = a \mid \mathcal H^{fil}_1 = \eta_1, \eta'_{1}, \omega'_{1} = \omega) \cdot 
\P_{\sigma}^h (\eta'_{1} \mid \mathcal H^{fil}_1 = \eta_1, \omega'_{1} = \omega) \\
=&
\sum_{\mathcal H^{rem}_1}
\P_{\sigma}^h (a'_{T_1} = a \mid \mathcal H^{fil}_1 = \eta_1, \mathcal H^{rem}_1, \omega'_{1} = \omega) \cdot 
\P_{\sigma}^h (\mathcal H^{rem}_1 \mid \mathcal H^{fil}_1 = \eta_1, \omega'_{1} = \omega).
\end{aligned}
\label{tralala4}
\end{equation}

Since $a'_{T_{1}}$ depends by definition only on the history of length $T_{1}$, we have
\begin{equation}
\P_{\sigma}^h (a'_{T_1} = a \mid \mathcal H^{fil}_1 = \eta_1, \mathcal H^{rem}_1, \omega'_{1} = \omega) = \P_{\sigma}^h (a'_{T_1} = a \mid \mathcal H^{fil}_1 = \eta_1, \mathcal H^{rem}_1).
\label{tralala5}
\end{equation}

By Lemma~\ref{aux-lemma}, we have
\begin{equation}
\P_{\sigma}^h (\mathcal H^{rem}_1 \mid \mathcal H^{fil}_1 = \eta_1, \omega'_{1} = \omega) = \P_{\sigma}^h (\mathcal H^{rem}_1 \mid \mathcal H^{fil}_1 = \eta_1).
\label{tralala6}
\end{equation}

From \eqref{tralala4}, \eqref{tralala5}, \eqref{tralala6}, and the law of total probability, we obtain
\begin{equation}
\begin{aligned}
\P_{\sigma}^h (a'_{T_1} = a \mid \mathcal H^{fil}_1 = \eta_1, \omega'_{1} = \omega) 
&=
\sum_{\mathcal H^{rem}_1}
\P_{\sigma}^h (a'_{T_1} = a \mid \mathcal H^{fil}_1 = \eta_1, \mathcal H^{rem}_1) \cdot 
\P_{\sigma}^h (\mathcal H^{rem}_1 \mid \mathcal H^{fil}_1 = \eta_1) \\
&=
\P_{\sigma}^h (a'_{T_1} = a \mid \mathcal H^{fil}_1 = \eta_1)
=
\widehat \sigma(\eta_1)(a).
\end{aligned}
\label{tralala7}
\end{equation}

From \eqref{tralala1}, \eqref{tralala2}, \eqref{tralala3}, and \eqref{tralala7}, we obtain
$$\P_{\widehat \sigma}^1(\eta_1, \omega_1 = \omega, a_1 = a) = \P_{\sigma}^h(\mathcal H^{fil}_1 = \eta_1, \omega'_{T_1} = \omega, a'_{T_1} = a). $$
This completes the proof of the base case for the induction.

We now suppose that \eqref{rooss10} holds for some $k \in \N^*$, and we will prove that it also holds for $k+1$. Let $\eta_{k+1} = (\eta_k, \overline a, \overline s).$

In $G_1$, we have (similarly to \eqref{tralala2})
\begin{equation} 
\P_{\widehat \sigma}^1(\eta_{k+1}, \omega_{k+1} = \omega, a_{k+1} = a) = \P_{\widehat \sigma}^1(\eta_{k+1}, \omega_{k+1} = \omega) \cdot \widehat \sigma(\eta_{k+1})(a).
\label{roo10}
\end{equation}

By construction of the events $\mathcal H^{fil}_{k+1} = \eta_{k+1}$ and $\eta_{k+1}$, we can write them as unions of disjoint events:

\begin{equation}
\begin{aligned}
\left\{\eta_{k+1}\right\} &= \bigcup_{\widehat \omega \in \Omega} \left\{\eta_{k}, \omega'_{T_{k}} = \widehat \omega, a'_{T_{k}} = \overline a, s'_{T_{k}+1} = \overline s \right\}; \\
\left\{\mathcal H^{fil}_{k+1} = \eta_{k+1}\right\} &= \bigcup_{\widehat \omega \in \Omega} \left\{\mathcal H^{fil}_{k} = \eta_{k}, \omega'_{T_{k}} = \widehat \omega, a'_{T_{k}} = \overline a, s'_{T_{k}+1} = \overline s \right\}.
\end{aligned}
\label{roo10x4}
\end{equation}

By \eqref{roo10x4}, the law of total probability, and the definition of conditional probability, we have
\begin{equation}
\begin{aligned}
\mathbb{P}_{\widehat{\sigma}}^1(\eta_{k+1}, \omega_{k+1} = \omega) 
&= \sum_{\widehat{\omega} \in \Omega} \mathbb{P}_{\widehat{\sigma}}^1(\eta_k, \omega_k = \widehat{\omega}, a_k = \overline{a}, \omega_{k+1} = \omega, s_{k+1} = \overline{s}) \\
&= \sum_{\widehat{\omega} \in \Omega} \mathbb{P}_{\widehat{\sigma}}^1(\eta_k, \omega_k = \widehat{\omega}, a_k = \overline{a}) \cdot \mathbb{P}_{\widehat{\sigma}}^1(\omega_{k+1} = \omega, s_{k+1} = \overline{s} \mid \eta_k, \omega_k = \widehat{\omega}, a_k = \overline{a}) \\
&= \sum_{\widehat{\omega} \in \Omega} \mathbb{P}_{\widehat{\sigma}}^1(\eta_k, \omega_k = \widehat{\omega}, a_k = \overline{a}) \cdot P(\omega \mid \overline{a}, \widehat{\omega}) \cdot \mathbf{1}_{f(\omega) = \overline{s}}
\end{aligned}
\label{roo10x1}
\end{equation}

In $G_h$, it follows from the definition of conditional probability that
\begin{multline}
\P_{\sigma}^h(\mathcal H^{fil}_{k+1} = \eta_{k+1}, \omega'_{T_{k+1}} = \omega, a'_{T_{k+1}} = a) = \\ \P_{\sigma}^h(\mathcal H^{fil}_{k+1} = \eta_{k+1}, \omega'_{T_{k+1}} = \omega) \cdot \P_{\sigma}^h(a'_{T_{k+1}} = a \mid \mathcal H^{fil}_{k+1} = \eta_{k+1}, \omega'_{T_{k+1}} = \omega).
\label{roo11}
\end{multline}

Using \eqref{roo10x3}, \eqref{roo10x4}, the law of total probability, and the definition of conditional probability, we obtain
\begin{equation}
\begin{aligned}
&\P_{\sigma}^h(\mathcal H^{fil}_{k+1} = \eta_{k+1}, \omega'_{T_{k+1}} = \omega) = \sum_{\widehat \omega \in \Omega}
\P_{\sigma}^h(\mathcal H^{fil}_{k} = \eta_{k}, \omega'_{T_{k}+1} = \omega, \omega'_{T_{k}} = \widehat \omega, a'_{T_{k}} = \overline a, s'_{T_{k}+1} = \overline s) \\
&=
\sum_{\widehat \omega \in \Omega}
\P_{\sigma}^h(\mathcal H^{fil}_{k} = \eta_{k}, \omega'_{T_{k}} = \widehat \omega, a'_{T_{k}} = \overline a) \cdot 
\P_{\sigma}^h(\omega'_{T_{k}+1} = \omega, s'_{T_{k}+1} = \overline s \mid \mathcal H^{fil}_{k} = \eta_{k}, \omega'_{T_{k}} = \widehat \omega, a'_{T_{k}} = \overline a). \\
&=
\sum_{\widehat \omega \in \Omega}
\P_{\sigma}^h(\mathcal H^{fil}_{k} = \eta_{k}, \omega'_{T_{k}} = \widehat \omega, a'_{T_{k}} = \overline a) \cdot P(\omega \mid \overline a, \widehat \omega) \cdot \mathbf{1}_{f(\omega) = \overline s}
\end{aligned}
\label{roo10x2}
\end{equation}

Thus, by \eqref{roo10x1}, \eqref{roo10x2}, and the induction hypothesis, we have
\begin{equation}
\mathbb{P}_{\widehat{\sigma}}^1(\eta_{k+1}, \omega_{k+1} = \omega) = \P_{\sigma}^h(\mathcal H^{fil}_{k+1} = \eta_{k+1}, \omega'_{T_{k+1}} = \omega)
\label{roo08}
\end{equation}

By the law of total probability, we obtain
\begin{multline}
\P_{\sigma}^h(a'_{T_{k+1}} = a \mid \mathcal H^{fil}_{k+1} = \eta_{k+1}, \omega'_{T_{k}+1} = \omega) = \\ 
\sum_{\substack{\text{Histories } \eta'_{k+1} \text{ of} \\ \text{length } T_{k+1} \text{ in } G_h}}
\P_{\sigma}^h(a'_{T_{k+1}} = a \mid \eta'_{k+1}, \mathcal H^{fil}_{k+1} = \eta_{k+1}, \omega'_{T_{k}+1} = \omega) \cdot \P_{\sigma}^h(\eta'_{k+1} \mid \mathcal H^{fil}_{k+1} = \eta_{k+1}, \omega'_{T_{k}+1} = \omega),
\label{roo04}
\end{multline}

Since $a'_{T_{k+1}}$ depends by definition only on the history of length $T_{k+1}$, we have

\begin{equation}
\P_{\sigma}^h(a'_{T_{k+1}} = a \mid \eta'_{k+1}, \mathcal H^{fil}_{k+1} = \eta_{k+1}, \omega'_{T_{k}+1} = \omega) = \P_{\sigma}^h(a'_{T_{k+1}} = a \mid \eta'_{k+1}).
\label{roo03}
\end{equation}

By the Bayes' theorem, we obtain
\begin{equation}
\P_{\sigma}^h(\eta'_{k+1} \mid \mathcal H^{fil}_{k+1} = \eta_{k+1}, \omega'_{T_{k}+1} = \omega) = 
\frac{\P_{\sigma}^h(\omega'_{T_{k}+1} = \omega \mid \eta'_{k+1}, \mathcal H^{fil}_{k+1} = \eta_{k+1}) \cdot \P_{\sigma}^h(\eta'_{k+1} \mid \mathcal H^{fil}_{k+1} = \eta_{k+1})}{\P_{\sigma}^h(\omega'_{T_{k}+1} = \omega \mid \mathcal H^{fil}_{k+1} = \eta_{k+1})}.
\label{roo01}
\end{equation}


Provided that the event 
$$\{\eta'_{k+1}, \mathcal H^{fil}_{k+1} = \eta_{k+1}\}$$
 has positive probability, we have $\eta'_{k+1} = \left(\mathcal{H}_{k+1}^{fil}, \mathcal{H}_{k+1}^{rem}\right)$. Consequently, we obtain
\begin{equation}
\P_{\sigma}^h(\omega'_{T_{k}+1} = \omega \mid \eta'_{k+1}, \mathcal H^{fil}_{k+1} = \eta_{k+1}) =
\P_{\sigma}^h(\omega'_{T_{k}+1} = \omega \mid \mathcal{H}_{k+1}^{rem}, \mathcal H^{fil}_{k+1} = \eta_{k+1})
\label{rttt02}
\end{equation}

By Lemma~\ref{aux-lemma}, we have

\begin{equation}
\P_{\sigma}^h(\omega'_{T_{k}+1} = \omega \mid \mathcal{H}_{k+1}^{rem}, \mathcal H^{fil}_{k+1} = \eta_{k+1}) =
\P_{\sigma}^h(\omega'_{T_{k}+1} = \omega \mid \mathcal H^{fil}_{k+1} = \eta_{k+1}),
\label{rttt01}
\end{equation}

Thus, by \eqref{roo01}, \eqref{rttt02}, and \eqref{rttt01}, we obtain

\begin{equation}
\P_{\sigma}^h(\eta'_{k+1} \mid \mathcal H^{fil}_{k+1} = \eta_{k+1}, \omega'_{T_{k}+1} = \omega) = 
\P_{\sigma}^h(\eta'_{k+1} \mid \mathcal H^{fil}_{k+1} = \eta_{k+1}).
\label{roo02}
\end{equation}

From \eqref{roo04}, \eqref{roo03}, \eqref{roo02}, and the law of total probability, we have
\begin{equation}
\begin{aligned}
\P_{\sigma}^h(a'_{T_{k+1}} = a \mid \mathcal H^{fil}_{k+1} = \eta_{k+1}, \omega'_{T_{k}+1} = \omega) &=  \sum_{\substack{\text{Histories } \eta'_{k+1} \text{ of} \\ \text{length } T_{k+1} \text{ in } G_h}}
\P_{\sigma}^h(a'_{T_{k+1}} = a \mid \eta'_{k+1}) \cdot \P_{\sigma}^h(\eta'_{k+1} \mid \mathcal H^{fil}_{k+1} = \eta_{k+1}) \\
&= \P_{\sigma}^h(a'_{T_{k+1}} = a \mid  \mathcal H^{fil}_{k+1} = \eta_{k+1}) = \widehat \sigma(\eta_{k+1})(a).
\end{aligned}
\label{roo09}
\end{equation}

Finally, by combining \eqref{roo10}, \eqref{roo11}, \eqref{roo08}, and \eqref{roo09}, we have
$$\P_{\widehat \sigma}^1(\eta_{k+1}, \omega_{k+1} = \omega, a_{k+1} = a) = 
\P_{\sigma}^h(\mathcal H^{fil}_{k+1} = \eta_{k+1}, \omega'_{T_{k+1}} = \omega, a'_{T_{k+1}} = a).$$

This completes the proof of the induction step.
\end{proof}

\subsection{Proof of Lemma~\ref{lemma3}}

\begin{proof}[Proof of Lemma~\ref{lemma3}]
We have
\begin{equation}
\begin{aligned}
\E_{\sigma}^h \left(\sum_{j=T_{k-1} + 1}^{T_k} g(\omega_j, a_j) \right) &= 
\E_{\sigma}^h \left(\sum_{m = 1}^{\infty} \mathbf{1}_{\{N_k \ge m\}} \cdot g(\omega_{T_{k-1} + m}, a_{T_{k-1} + m})\right) \\
&= \sum_{m = 1}^{\infty}  \E_{\sigma}^h \left(\mathbf{1}_{\{N_k \ge m\}} \cdot g(\omega_{T_{k-1} + m}, a_{T_{k-1} + m})\right),
\end{aligned}
\label{hoho1}
\end{equation}
where the last equality holds by Fubini's theorem (since $g$ is bounded and $\E(N_i) = 1/h < \infty$, the tail sum formula implies absolute convergence). We also have
\begin{equation}
\begin{aligned}
\E_\sigma^h g(\omega_{T_k}, a_{T_k}) &= 
\E_{\sigma}^h \left(\sum_{m = 1}^{\infty} \mathbf{1}_{\{N_k = m\}} \cdot g(\omega_{T_{k-1} + m}, a_{T_{k-1} + m})\right) \\
&= \sum_{m = 1}^{\infty} \E_{\sigma}^h \left( \mathbf{1}_{\{N_k = m\}} \cdot g(\omega_{T_{k-1} + m}, a_{T_{k-1} + m})\right).
\end{aligned}
\label{hoho2}
\end{equation}

Now, the event $\{N_k = m\}$ can be expressed as the intersection of two independent events: 
$$\{N_k = m\} = \{N_k\ge m\} \cap \{X_{T_{k-1}+m} = 1\}.$$
Note that the random variable $X_{T_{k-1}+m}$ depends only on $h$ and is entirely independent of $N_i$, as well as the states and actions up to stage $T_{k-1}+m$. Consequently, the random variable $\mathbf{1}_{\{X_{T_{k-1}+m} = 1\}}$ is independent of the  random variable $\mathbf{1}_{\{N_k \ge m\}} \cdot g(\omega_{T_{k-1} + m}, a_{T_{k-1} + m})$. Hence, we have 
\begin{equation}
\begin{aligned}
\E_{\sigma}^h \left(\mathbf{1}_{\{N_k = m\}} \cdot g(\omega_{T_{k-1} + m}, a_{T_{k-1} + m})\right) &= \E_{\sigma}^h \left( \mathbf{1}_{\{X_{T_{k-1}+m} = 1\}} \cdot \mathbf{1}_{\{N_k \ge m\}} \cdot g(\omega_{T_{k-1} + m}, a_{T_{k-1} + m})\right) \\
&= \E_{\sigma}^h \left( \mathbf{1}_{\{X_{T_{k-1}+m} = 1\}}\right) \cdot \E_{\sigma}^h \left( \mathbf{1}_{\{N_k \ge m\}} \cdot g(\omega_{T_{k-1} + m}, a_{T_{k-1} + m})\right) \\
&= \P(X_{T_{k-1}+m} = 1) \cdot \E_{\sigma}^h \left( \mathbf{1}_{\{N_k \ge m\}} \cdot g(\omega_{T_{k-1} + m}, a_{T_{k-1} + m})\right) \\
&= h \cdot \E_{\sigma}^h \left( \mathbf{1}_{\{N_k \ge m\}} \cdot g(\omega_{T_{k-1} + m}, a_{T_{k-1} + m})\right).
\end{aligned}
\label{hoho21}
\end{equation}

The lemma now follows directly from \eqref{hoho1}, \eqref{hoho2}, and \eqref{hoho21}.
\end{proof}

\subsection{Proof of Lemma~\ref{lemma4}}

\begin{proof}[Proof of Lemma~\ref{lemma4}]
Let
$M := \max_{\omega, a} |g(\omega, a)|.$
By the triangle inequality, we have 
\begin{equation}
\E_{\sigma}^h \left(\left| \sum_{j = 1}^{T_k} g(\omega_j, a_j) - \sum_{j = 1}^{t_k} g(\omega_j, a_j)\right|\right) \le M \cdot \E_{\sigma}^h \big( |T_k - t_k| \big).
\label{rmx001}
\end{equation}

By Jensen's inequality, it follows that
\begin{equation}
\E_{\sigma}^h \left(\left|T_k - \frac k h \right| \right) = \E_{\sigma}^h \left(\left|T_k - \E_{\sigma}^h (T_k) \right|\right) \le \sqrt{\E_{\sigma}^h \left(\left|T_k - \E_{\sigma}^h (T_k) \right|^2\right)} =  \sqrt{\operatorname{Var}(T_k)} = \sqrt{\frac{1-h}{h^2} \cdot k }.
\label{rmx002}
\end{equation}

Finally, we have
\begin{equation}
\begin{aligned}
&\left|\E_{\sigma}^h\left( \frac{h}{k} \sum_{j = 1}^{T_k} g(\omega_j, a_j) 
-
\frac{1}{t_k} \sum_{j = 1}^{t_k} g(\omega_j, a_j) \right)\right| \\
= 
&\left|\E_{\sigma}^h\left( \frac{h}{k} \sum_{j = 1}^{T_k} g(\omega_j, a_j) 
-
\frac{h}{k} \sum_{j = 1}^{t_k} g(\omega_j, a_j)
+
\frac{h}{k}\sum_{j = 1}^{t_k} g(\omega_j, a_j)
-
\frac{1}{t_k} \sum_{j = 1}^{t_k} g(\omega_j, a_j) \right)\right| \\
\le
&\left|\E_{\sigma}^h \left[ \frac{h}{k} \left( \sum_{j = 1}^{T_k} g(\omega_j, a_j) 
-
\sum_{j = 1}^{t_k} g(\omega_j, a_j)\right) \right] \right|
+
\left| \E_{\sigma}^h \left(\left( \frac{h}{k} - \frac{1}{t_k}\right) \sum_{j = 1}^{t_k} g(\omega_j, a_j) \right) \right|,
\end{aligned}
\label{rmx003}
\end{equation}

We now evaluate each term of the above sum. By \eqref{rmx001} and \eqref{rmx002}, we have
\begin{equation}
\begin{aligned}
& \left|\E_{\sigma}^h \left[ \frac{h}{k} 
\left( \sum_{j = 1}^{T_k} g(\omega_j, a_j) 
-
\sum_{j = 1}^{t_k} g(\omega_j, a_j)\right) \right] \right|
\le
\frac{h}{k} \E_{\sigma}^h \left( \left|
\left( \sum_{j = 1}^{T_k} g(\omega_j, a_j) 
-
\sum_{j = 1}^{t_k} g(\omega_j, a_j)\right) \right|\right) \\
\le &
\frac{M h}{k} \E_{\sigma}^h \big( \left|T_k - t_k \right| \big)
\le
\frac{M h}{k} \E_{\sigma}^h \left(\left|T_k - \frac{k}{h}\right| \right)
+
\frac{M h}{k} \left|t_k - \frac{k}{h}\right| 
= M \cdot \sqrt{\frac{1-h}{k}} + 
\frac{M h}{k} \cdot 1 \xrightarrow{k \to \infty} 0.
\end{aligned}
\label{rmx004}
\end{equation}

We have
\begin{equation}
\begin{aligned}
\left| \E_{\sigma}^h \left( \left( \frac{h}{k} - \frac{1}{t_k}\right) \sum_{j = 1}^{t_k} g(\omega_j, a_j) \right) \right|
\le 
\left| \frac{h}{k} - \frac{1}{t_k}\right| \cdot \E_{\sigma}^h \left( \left| \sum_{j = 1}^{t_k} g(\omega_j, a_j)  \right| \right)
&\le 
\left| \frac{h}{k} - \frac{1}{t_k}\right| \cdot M t_k \\
&= 
M \left| \frac{h}{k} t_k - 1\right| \xrightarrow{k \to \infty} 0.
\end{aligned}
\label{rmx005}
\end{equation}

The lemma now follows from \eqref{rmx003}, \eqref{rmx004}, \eqref{rmx005}, and the standard fact that $\lim(x_n-y_n) = 0$ implies $\liminf x_n = \liminf y_n$.
\end{proof}

\subsection{Proof of Lemma~\ref{lemma5}}

\begin{proof}[Proof of Lemma~\ref{lemma5}]
Since $\{x_{n_k}\}_{k=1}^\infty$ is a subsequence of $\{x_n\}_{n=1}^\infty$, we have 
\begin{equation}
\liminf_{n \to \infty} x_n \le \liminf_{k \to \infty} x_{n_k}.
\label{kkk1}
\end{equation}

We now prove the reverse inequality. By definition, there exists a subsequence $\{x_{m_k}\}_{k=1}^\infty$ converging to $\liminf_{n \to \infty} x_n$. For each subsequence index $m_j$ of $\{x_{m_k}\}_{k=1}^\infty$, let $n_{k_j}$ be the largest subsequence index of $\{x_{n_k}\}_{k=1}^\infty$ such that $n_{k_j} \le m_j$. By the triangle inequality, we have
\begin{equation}
|x_{m_j} - x_{n_{k_j}}| \le \sum_{i = n_{k_j}}^{m_j - 1} |x_{i+1} - x_i| \le (m_j - n_{k_j}) \cdot \sup_{i \ge n_{k_j}} |x_{i+1} - x_i| \le M \cdot \sup_{i \ge n_{k_j}} |x_{i+1} - x_i| \xrightarrow{j \to \infty} 0.
\label{kkk2}
\end{equation}

Since $n_{k_j} \to \infty$ as $j \to \infty$, the set of accumulation points of the sequence $\{x_{n_{k_j}}\}_{j=1}^\infty$  is contained in the set of accumulation points of the sequence $\{x_{n_{k}}\}_{k=1}^\infty$. Together, this inclusion and \eqref{kkk2} imply
\begin{equation}
\liminf_{k \to \infty} x_{n_k} \le \liminf_{j \to \infty} x_{n_{k_j}} = \liminf_{j \to \infty} x_{m_j} = \lim_{j \to \infty} x_{m_j} = \liminf_{n \to \infty} x_n,
\label{kkk3}
\end{equation}
where the last equality follows from the definition of the subsequence $\{x_{m_j}\}_{j=1}^\infty$. The lemma now follows directly from \eqref{kkk1} and \eqref{kkk3}.
\end{proof}

\section{Concluding remarks}
\subsection{Fully observed state}
\label{final1}

This subsection is dedicated to the case where the signal fully reveals the state, meaning that $S = \Omega$ and $f(\omega) = \omega$.  In such a case, \cite[Proposition 5.2]{SorVig16} proves that for all $\lambda \in (0,1]$ and $h \in (0,1]$, we have
$$V_\lambda(h) = V_{\frac{\lambda}{1+\lambda - \lambda h}}(1).$$

This implies that for any $h \in (0,1]$, we have
$$V(h) = \lim_{\lambda \to 0} V_\lambda(h) =  \lim_{\lambda \to 0} V_{\frac{\lambda}{1+\lambda - \lambda h}}(1) = \lim_{\lambda \to 0} V_\lambda(1) = V(1).$$

Hence, $V$ is a constant function in this particular case.

\subsection{Continuity of $V(h)$}
\label{final2}

In this subsection, we examine whether the function $h \mapsto V(h)$ is continuous.

\begin{proposition}
\label{graa1}
The function $h \mapsto V(h)$ is lower semi-continuous on $(0,1)$.
\end{proposition}

\begin{proof}
If $h' \in (0,1)$ and $\varepsilon > 0$ is sufficiently small, then the support of the transition probabilities of the POMDPs in the family $\{G_h\}_{h \in (h'-\varepsilon, h'+\varepsilon)}$ is identical. The proposition now follows directly from \cite[Corollary~3.10]{Cha22}.
\end{proof}

\begin{proposition}
\label{graa2}
The function $h \mapsto V(h)$ is left-continuous on $(0,1)$.
\end{proposition}

\begin{proof}
This follows directly from Proposition~\ref{graa1} and from the monotonicity of $V(h)$.
\end{proof}

However, it is possible to introduce new transitions when moving from $h=1$ to $h<1$, which implies that the results from \cite{Cha22} cannot be applied. In fact, it is possible that $V(h)$ is not lower semi-continuous at $1$, even in the state-blind case. 

\begin{example}[An example of a POMDP in which $V(h)$ is neither lower semi-continuous nor left-continuous at~$1$]
\label{exam1}
Consider a POMDP $G_1$ with the following components:
\begin{itemize}
	\item Signal set: $\{s_1\}$.
	\item Action set: $\{a, b\}$.
	\item State set: $\{\omega_1, \omega_2, \omega_3\}$.
	\item Initial state: $\omega_1$.
	\item Stage payoff function: $g(a,\omega_1) = g(b,\omega_1) = g(a,\omega_2) = g(b,\omega_2) = 1$, and $g(a,\omega_3) = g(b,\omega_3) = 0$.
	\item Transitions: $$P(\omega_1 \mid b,\omega_2) = P(\omega_2 \mid a,\omega_1) = P(\omega_3 \mid b, \omega_1) = P(\omega_3 \mid a,\omega_2) = P(\omega_3 \mid a,\omega_3) = P(\omega_3 \mid b,\omega_3) = 1.$$
\end{itemize}
See Figure~\ref{fig1} for a visual representation. 

Now, given a base POMDP  $G_1$, consider the POMDP with stage duration $G_h$. It is clear that the decision maker can achieve a payoff of $1$ by playing the strategy $(a,b,a,b, \ldots)$. However, as soon as $h<1$, the state stays frozen with a probability of $1-h$. Because of this, the decision maker's belief about the state (conditional on the fact that it is not $\omega_3$) will converge to $\frac 1 2 \omega_1 + \frac 1 2 \omega_2$. 
Since there is only a single signal ($s_1$), the decision maker cannot detect when a state transition fails to occur. Once his belief gets sufficiently close to
$\frac{1}{2}\omega_1 + \frac{1}{2}\omega_2$, any chosen action carries a probability of $\approx 0.5$ of incorrectly guessing the current state. Because of that, the decision maker eventually reaches the absorbing state $\omega_3$ with payoff $0$. Thus we have
$$V(h) = 
\begin{cases}
1, \text{ if } h = 1 \\
0, \text{ if } h < 1.
\end{cases}$$
\end{example}

\begin{figure}[h]
	\large
    \centering 
    
    \begin{tikzpicture}[
        ->,                 
        >=stealth,          
        shorten >=1pt,      
        auto,               
        node distance=3cm, 
        very thick,         
        state/.style={      
            circle, 
            draw, 
            minimum size=1.0cm,
            font=\Large     
        }
    ]

    \node[state, label=above:$\omega_1$] (s1) at (0,0) {1};
    
    \node[state, label=above:$\omega_2$] (s2) at (4,0) {1};
    
    \node[state, label=below:$\omega_3$] (s3) at (2,-3) {0};

    \path
        (s1) edge [bend left=20] node {$a$} (s2)
        
        (s2) edge [bend left=20] node {$b$} (s1)
        
        (s1) edge [bend right=15] node [swap] {$b$} (s3)
        
        (s2) edge [bend left=15] node {$a$} (s3)
        
        (s3) edge [loop, out=245, in=215, looseness=8] node [below left] {$a$} (s3)
        (s3) edge [loop, out=295, in=325, looseness=8] node [below right] {$b$} (s3);
    \end{tikzpicture}
    \caption{The POMDP $G_1$ from Example~\ref{exam1}}
    \label{fig1}
\end{figure}

\begin{remark}
It is not yet known whether $h \mapsto V(h)$ is continuous on $(0,1)$. In general, it is possible for the asymptotic value to be discontinuous even if no new transitions are introduced; see \cite[Proposition~3.11 and its proof]{Cha22}. However, the author was unable to adapt the counterexample from \cite{Cha22} to construct one for $V(h)$. \demo
\end{remark}

\subsection{The case of random signals}
\label{grana1}

Throughout the paper, we have assumed that the signal is deterministic and is given by a function $f : \Omega \to S$. We can examine what happens if the signal is random, i.e., if it is given by a function $f : \Omega \to \Delta(S)$. In this case, we can still consider POMDPs with stage duration, as defined in Definition~\ref{klp01}. However, our proof no longer works because Lemma~\ref{aux-lemma} no longer holds. Moreover, Theorem~\ref{mainlemma} is no longer valid. To see this, we consider the following example.

\begin{example}[An example of a POMDP with random signals where Theorem~\ref{mainlemma} does not hold]
\label{exam2}
Consider a POMDP $G_1$ with the following components:
\begin{itemize}
	\item Signal set: $\{s_1, s_2\}$.
	\item Action set: $\{a, b\}$.
	\item State set: $\{\omega_1, \omega_2\}$.
	\item Signal-giving function: $f(\omega_1) = s_1$ and $f(\omega_2) = \frac 1 2 s_1 + \frac 1 2 s_2$.
	\item Initial state: $\frac 1 2 \omega_1 + \frac 1 2 \omega_2$.
	\item Stage payoff function: $g(a,\omega_1) = g(b,\omega_2) = 1$, and $g(b,\omega_1) = g(a,\omega_2) = 0$.
	\item Transitions: $P(\omega_1 \mid \cdot,\cdot) = P(\omega_2 \mid \cdot, \cdot) = 1 / 2.$
\end{itemize}
See Figure~\ref{fig2} for a visual representation.

Let us show that $\lim_{h \to 0} V(h) > V(1)$, which implies that the function $h \mapsto V(h)$ is not nondecreasing. From the structure of the POMDP, we see that an optimal strategy in $G_1$ is to play $a$ if the observed signal is $s_1$, and to play $b$ if the observed signal is $s_2$. Under this strategy, the expected stage payoff is 
\[
\underbrace{\frac 3 4}_{\P(s_1)} \cdot \underbrace{\frac 2 3}_{\P(\omega_1 \mid s_1)} + \underbrace{\frac 1 4}_{\P(s_2)} \cdot \underbrace{1}_{\P(\omega_2 \mid s_2)} = \frac 3 4.
\]
Thus the optimal total payoff in $G_1$ is $3/4$. 

In $G_h$, the state does not move for $\approx 1 / h$ stages. Consequently, for a small $h$, the decision maker essentially observes the underlying signal distribution. As a result, the optimal payoff converges to $1$ as $h$ vanishes.

Thus, we have shown that
\[\lim_{h \to 0} V(h) = 1 > 3/4 = V(1).\]
\end{example}

\begin{figure}[H]
    \centering 
    
    \begin{tikzpicture}[
        ->,                 
        >=stealth,          
        shorten >=1pt,      
        auto,               
        node distance=5cm, 
        very thick,         
        state/.style={      
            circle, 
            draw, 
            minimum size=1.2cm,
            font=\Large     
        },
        info/.style={
            align=center,
            font=\normalsize
        }
    ]

    \node[state] (w1) at (0,0) {$\omega_1$};
    \node[state] (w2) at (6,0) {$\omega_2$};
    
    \node[info, below=0.3cm of w1] {
        \textbf{Signal:} $s_1$ (100\%) \\ 
        $g(a, \omega_1) = 1$ \\ 
        $g(b, \omega_1) = 0$
    };
    
    \node[info, below=0.3cm of w2] {
        \textbf{Signal:} $s_1$ (50\%), $s_2$ (50\%) \\ 
        $g(a, \omega_2) = 0$ \\ 
        $g(b, \omega_2) = 1$
    };

    \path
        (w1) edge [loop left, looseness=6] node [left] {$1/2$} (w1)
        (w1) edge [bend left=20] node [above] {$1/2$} (w2)
        
        (w2) edge [bend left=20] node [below] {$1/2$} (w1)
        (w2) edge [loop right, looseness=6] node [right] {$1/2$} (w2);
        
    \end{tikzpicture}
    \caption{The POMDP $G_1$ from Example~\ref{exam2}}
    \label{fig2}
\end{figure}

\section{Acknowledgments}

The author is grateful to Guillaume Vigeral for his help during the writing of this article. The author is grateful to Raimundo Saona and Eilon Solan for useful discussions.

\bibliographystyle{chicago}
\bibliography{ref}

\end{document}